\documentclass[10pt,reqno]{amsart}
\usepackage{amsfonts,amssymb}
\theoremstyle{plain}

\newtheorem{theorem} {Theorem}[section]
\newtheorem{pro}{Proposition}[section]

\newtheorem{definition}{Definition}[section]
\newtheorem{lemma}[theorem]{Lemma}

\newtheorem{rem}{Remark}[section]
\newcommand{\mat}{\mathbb}
\newcommand{\cal}{\mathcal}

\newcommand{\FF}{\mathcal{F}}
\newcommand{\E}{\mathbb{E}}
\newcommand{\N}{\mathbb{N}}
\newcommand{\LL}{\mathbb{L}}

\newcommand{\R}{\mathbb{R}}

\newcommand{\PP}{\mathbb{P}}

\begin{document}
	\title[ Asymptotically periodic solution of a Sto. Diff. Eq. ]{On asymptotically periodic  solution of a  Stochastic Differential  equation}

\maketitle{}

\centerline{Solym Mawaki MANOU-ABI* $^{\rm 1,2}$, William DIMBOUR $^{\rm 3}$}

\vspace{12pt}

\centerline{$^{\rm 1}$ CUFR  de Mayotte}
\centerline{D\'epartement Sciences et Technologies}
\centerline{solym.manou-abi@univ-mayotte.fr }
\centerline{Phone number : (+33) (0) 7 51 48 36 63}

\vspace{12pt}

\centerline{$^{\rm 2}$ Institut Montpelli\'erain Alexander Grothendieck}
\centerline{UMR CNRS 5149, Universit\'e de Montpellier}
\centerline{solym-mawaki.manou-abi@umontpellier.fr}

\vspace{12pt}

\centerline{$^{\rm 3}$ UMR Espace-Dev, Universit\'e de Guyane}
\centerline{Campus de Troubiran 97300}
\centerline{ Cayenne  Guyane (FWI)}
\centerline{william.dimbour@espe-guyane.fr}

\renewcommand{\thefootnote}{}
\footnote{*Corresponding author}
\renewcommand{\thefootnote}{\arabic{footnote}}
\setcounter{footnote}{0}
\date{}
\maketitle

\begin{abstract}
In this paper, we first introduce the concept and properties  of  $\omega$-periodic limit process. Then we apply specific criteria obtained to investigate  asymptotically $\omega$-periodic mild solutions of  a Stochastic Differential Equation  driven by a Brownian motion. 
Finally, we give an  example to show  usefulness of the theoritical results that we obtain  in the paper.   
\end{abstract}
\vspace{12pt}



\section{Introduction}
The  recurrence of dynamics for stochastic and deterministic processes produced  by many different kinds of stochastic and deterministic  equations is one of the most important topics in the qualitative theory of stochastic processes and functions, due both to its mathematical interest and its applications in many scientific fields, such as mathematical biology, celestial mechanics, non linear vibration, control theory, to name few. The concept of periodicity was studied for dynamics of  stochastic processes and functions.  However  the dynamics observed  in  some phenomena in the real world are not periodic, but  almost approximately or asymptotically periodic; see for instance  \cite{cordu,kundert,ahmad} for  almost periodic observations.\\
 
In the past several decades many authors suggested and developed several extensions of the concept of periodicity, in the deterministic and stochastic case, such as the  almost automorphy, pseudo almost periodicity, asymptotically periodicity, etc.
  (see \cite{bez, beza, bezan, manou, cao, chang, cuevas, dimbour, dm, sun, zhi1, zhi2, xie2, zhang}  and references therein)\\
 
Recently, the concept of  periodic limit function has been introduced by Xie and Zhang \cite{xie1}  to generalize the notion of asymptotic periodicity. The authors  investigate some properties of periodic  limit functions in order to study the existence and uniqueness of asymptotically periodic solutions of some differential equations. However, to the best of our knowledge, there is no work or applicable results for stochastic  differential equations.Therefore, in this paper, we will  introduce  the notion of square mean periodic limit process. Then we'll investigate their qualitative properties   in  order to study the existence of square mean asymptotically periodic mild solution to the following Stochastic Differential Equation (SDE) driven by Brownian motion : 

\begin{equation*}
\label{eq: base}
\left\{
\begin{array}{l}
dX(t) =  A X(t)dt + f(t,X(t))dt  + g(t,X(t))dB(t), \quad \quad t \geq 0 \\
X(0) =  c_{0},
\end{array}
\right.
\end{equation*}
where  $ c_{0} \in \LL^{2}(\PP,\mat{H})$ and $A$ is an infinitesimal generator which generates a $C_0$ semigroup, denoted by $(T(t))_{t\geq 0}$. 
In addition,  $$   f: \R \times \LL^{2}(\PP, \mat{H}) \rightarrow  \LL^{2}(\PP, \mat{H}),    $$

$$  g: \R \times \LL^{2}(\PP, \mat{H}) \rightarrow  \LL^{2}(\PP, \mat{H})  $$
are   Lipschitz continuous and  bounded, and $B(t)$  is a two-sided  standard one-dimensional Brownian motion, which is defined on  the filtered complete probability space $(\Omega,\mathcal{F},\cal{F}_{t},\mat{P})$   with values in the separable Hilbert space  $\mat{H}$.  Here  $\cal{F}_{t}=\sigma\{ B(u)-B(v) / u, v \leq t \}$.

The paper is organised as follows:  In  Section 2, we preliminarily introduce the space of  square-mean $\omega$-periodic limit process and study properties of such processes. It also includes some results, not only on the completeness of the space that consists of the square-mean $\omega$-periodic limit processes but also on the composition of such processes.  Based on the results in Section 2 and given some suitable conditions, we prove in Section 3 the existence as well as the uniqueness of the square-mean  asymptotically $\omega$-periodic solution  for the  above SDE.  Finally, an illustrative example is provided to show the feasibility of the theoretical results developed in the paper.

\section{Square mean omega-periodic limit process}
This section is concerned with some notations, definitions, lemmas and preliminary facts that  may be used
in what follows.  Throughout this paper we consider a real separable Hilbert space  $(\mat{H},||.||)$ and
a probability space $(\Omega, \mathcal{F},\mathbb{P})$  equipped with a filtration $(\mathcal{F}_{t})_{t}$.  Denote by  $\LL^{2}(\PP,\mat{H})$  the space of all strongly measurable square integrable $\mat{H}$-valued random variables such that
$$ \E ||X||^{2} = \int_{\Omega} ||X(\omega)||^{2}d\PP(\omega) < \infty.  $$
For $X \in \LL^{2}(\PP,\mat{H})$, let $||X||_{2} =  (\E ||X||^{2})^{1/2}$. 
Then it is routine to check that $\LL^{2}(\PP,\mat{H})$ is a Hilbert space equipped with the norm $||.||_{2}$.


\begin{definition}
A stochastic process $X : \R_+ \rightarrow \LL^{2}(\PP, \mat{H})$ is said to be continuous whenever 
$$  \lim_{t\rightarrow s} \E || X(t)-X(s)||^{2} = 0.       $$
\end{definition} 

\begin{definition}
A stochastic process $X : \R_+ \rightarrow \LL^{2}(\PP, \mat{H})$ is said to be bounded if 
there exists a constant $K > 0$ such that $$ \E ||X(t)||^{2} \leq K \quad \forall t \geq 0 $$
\end{definition} 
By $CUB \big(\R_{+},\LL^{2}(\PP, \mat{H})\big)  $ we denote the collection of all continuous
 and uniformly bounded stochastic processes from $\R_{+} \rightarrow \LL^{2}(\PP, \mat{H}) $.

\begin{definition}
\label{def3}
A continous  and  bounded stochastic process $X : \R_{+} \rightarrow \LL^{2}(\PP, \mat{H})$ is said to 
be square mean $\omega$-periodic limit  if there exists $\omega >0$ such that 
 $$\lim_{n \rightarrow +\infty } \E || X(t+n\omega)-\tilde{X}(t)||^{2} = 0$$
 is well defined for each $t\geq 0$ when $n\in \mathbb{N}$ for some stochastic process\\ 
 $\tilde{X} : \R_{+} \rightarrow  \LL^{2}(\PP, \mat{H})$.
 \end{definition} 
 
 \begin{rem}  
For all $ t \geq 0$,  $\tilde{X}(t) $ is the limit of $ X(t+n\omega)$ in $\LL^{2}(\PP, \mat{H})$ when 
$n \to + \infty$, when it exists. 
The collection of such  $\omega$-periodic  limit processes is denoted by $P_{\omega}L\big(\R_{+},\LL^{2}(\PP, \mat{H})\big) $.
Note also that  the process $\tilde{X}$ in the previous definition is measurable but not neccessarily continuous.
\end{rem}
In the following Proposition, we list some properties of square mean $\omega$-periodic limit process. 

\begin{pro}
\label{proposition}
Let $X$ be square mean $\omega$-periodic limit process such that 
$$\lim_{n \rightarrow +\infty } \E || X(t+n\omega)-\tilde{X}(t)||^{2} = 0$$
 is well defined for each $t\geq 0$ when $n\in \mathbb{N}$ for some stochastic process\\
 $\tilde{X} : \R_{+} \rightarrow  \LL^{2}(\PP, \mat{H})$. \\ If $X, X_{1}, X_{2}$ 
 are square mean $\omega$-periodic limit processes  then the following are true :
 \begin{itemize}
 \item[(a)] $X_{1} + X_{2}$ is  a square mean $\omega$-periodic limit process.
 \item[(b)] $cX$ is a square mean $\omega$-periodic limit process for every scalar $c$. 
 \item[(c)]  We have $$ \E || \tilde{X}(t+\omega)-\tilde{X}(t)||^{2} = 0.$$
 \item[(d)]  $\tilde{X}$ is bounded on $\R_+$ and $||\tilde{X}||_{\infty} \leq ||X||_{\infty} \leq K$. 
 \item[(e)] $X_{a}(t) = X(t+a)$ is a square mean  $\omega$-periodic limit process for \\
 each fixed $a \in \R_+$.
 \end{itemize}
\end{pro}

 \begin{proof}
 The proof is  straightforward but we will only prove the statement in $(c)$. \\
 To this end, note that for each $n \geq 1$, we have : 
  \begin{align*}
0 \leq \E \left|\left| \tilde{X}(t+\omega)- \tilde{X}(t)\right|\right|^{2}   
 & \leq 2 \E \left|\left| \tilde{X}(t+\omega)- X(t+(n+1)\omega)\right|\right|^{2}  \\
 & + 2 \E \left|\left| X(t+(n+1)\omega)- \tilde{X}(t)\right|\right|^{2},\\
\end{align*}
Let $\epsilon >0 $,  for $N$ sufficiently large if $n \geq N$ then 
  $$\E \left|\left| \tilde{X}(t+\omega)- X(t+(n+1)\omega)\right|\right|^{2} \leq \epsilon/2$$ and 
  $$\E \left|\left| X(t+(n+1)\omega)- \tilde{X}(t)\right|\right|^{2} \leq \epsilon/2.$$
  so that  $ \E \left|\left| \tilde{X}(t+\omega)- \tilde{X}(t)\right|\right|^{2}   \leq \epsilon.$
   Thus 
 $ \E \left|\left| \tilde{X}(t+\omega)- \tilde{X}(t)\right|\right|^{2}   = 0.$
 \end{proof}
 
 Because of the above Proposition, we give name of $\omega$-periodic limit process in Definition \ref{def3}.

 \begin{rem}
Note that if 
$$ \E ||\tilde{X}(t+\omega)- \tilde{X}(t)||^{2} = 0  $$ then 
$$ \E ||\tilde{X}(t+p \omega)- \tilde{X}(t)||^{2}  =0 \;  \textrm{for all } p \geq 1.   $$
\end{rem}

\begin{theorem}
The space $P_{\omega}L\big(\R_{+},\LL^{2}(\PP, \mat{H})\big) $ is a Banach space  equipped with the norm 
$ ||X ||_{\infty} = \sup_{t \geq 0} \big( \E ||X(t)||^{2}\big) ^{1/2} =  \sup_{t \geq 0}  ||X(t)||_{2}.$
\end{theorem}

\begin{proof}
By Proposition \ref{proposition},  $P_{\omega}L\big(\R_{+},\LL^{2}(\PP, \mat{H})\big) $ is a vector space, then it is easy to verify that 
$ || .||_{\infty}$ is a norm on $P_{\omega}L\big(\R_{+},\LL^{2}(\PP, \mat{H})\big) $. 
We only need to show that $P_{\omega}L\big(\R_{+},\LL^{2}(\PP, \mat{H})\big) $ is complete with respect to the norm 
$ || .||_{\infty} $. To this end, assume that  $(X_{n})_{n\geq 0}  \in  P_{\omega}L\big(\R_{+},\LL^{2}(\PP, \mat{H})\big) $
is a Cauchy  sequence with respect to $||.||_{\infty}$ and that $X$ is the pointwise limit of $X_n$ with respect to  $||.||_{2}$ ; i.e. 
\begin{equation}
\label{eq:cauchy1}
\lim_{n\rightarrow +\infty }  \E || X_{n}(t)-X(t)||^{2} = 0
\end{equation}

for each $t \geq 0$. 
Note that this limit $X$  always exists by the completeness of $\LL^{2}(\PP, \mat{H})$ with respect to $||.||_{2}$.\\
Since  $(X_{n})_{n\geq 0} $ is Cauchy with respect to $ ||.||_{\infty} $ the convergence in (\ref{eq:cauchy1}) is uniform for $t\geq 0$. We need to show that
$X \in P_{\omega}L\big(\R_{+},\LL^{2}(\PP, \mat{H})\big)$. 
 First note that $X$ is stochastically  continuous from the uniform convergence of $X_n$ to $X$ with respect to $||.||_{2}$ and the stochastic continuity of  $X_n$. 
 Next we  prove that $X$ is  a square mean $\omega$-periodic limit process. 
 By the definition of $(X_{n})_{n\geq 0} \in  P_{\omega}L\big(\R_{+},\LL^{2}(\PP, \mat{H})\big) $ we have for all $i\geq 0$, : 
 
 \begin{equation}
 \label{eq:cauchy2}
  \lim_{n \rightarrow +\infty } \E || X_{i}(t+n\omega)-\tilde{X}_{i}(t)||^{2} = 0,
 \end{equation}
 for some stochastic process $\tilde{X} _{i} : \R_{+} \rightarrow \LL^{2}(\PP, \mat{H})\big) $. 
 Let us point out  that for each $t \geq 0$, the sequence 
 $(\tilde{X}_{i}(t))_{i\geq 0} $ is a Cauchy sequence in $\LL^{2}(\PP, \mat{H})\big)$.  Indeed, we have

  \begin{align*}
\E \left|\left| \tilde{X}_{i}(t)- \tilde{X}_{k}(t)\right|\right|^{2}
 & \leq 3 \E \left|\left| \tilde{X}_{i}(t)- X_{i}(t+n\omega)\right|\right|^{2}  \\
 & + 3 \E \left|\left| X_{i}(t+n\omega)- X_{k}(t+n\omega)\right|\right|^{2}\\
 & + 3 \E \left|\left| X_{k}(t+n\omega)- \tilde{X}_{k}(t)\right|\right|^{2}.\\
\end{align*}
By  (\ref{eq:cauchy2})  and the fact that the sequence   $(X_{n})_{n\geq 0}  $ is Cauchy, we get that the sequence  $(\tilde{X}_{i}(t))_{i\geq 0} $  is Cauchy.\\

Using the completness of the space $\LL^{2}(\PP, \mat{H})\big)$, we denote by $\tilde{X}$ the pointwise limit of $(\tilde{X}_{i}(t))_{i\geq 0} $ such that 

\begin{equation}
\label{eq:cauchy3}
\lim_{i\rightarrow +\infty }  \E || \tilde{X}_{i}(t)-\tilde{X}(t)||^{2} = 0.
\end{equation}

Let us prove now that 
$$   \lim_{n \rightarrow +\infty } \E || X(t+n\omega)-\tilde{X}(t)||^{2} = 0,$$
for each $t \geq 0$.  Indeed for each $i \geq 0 $, we have 

  \begin{align*}
\E \left|\left| X(t+n\omega)-  \tilde{X}(t)\right|\right|^{2}
 & \leq 3  \E \left|\left| X(t+n\omega)- X_{i}(t+n\omega) \right|\right|^{2}\\
 & + 3 \E \left|\left| X_{i}(t+n\omega)-   \tilde{X}_{i}(t)  \right|\right|^{2}\\
 & + 3 \E \left|\left|  \tilde{X}_{i}(t) - \tilde{X}(t)\right|\right|^{2}.\\
\end{align*}
By  (\ref{eq:cauchy1}) , (\ref{eq:cauchy2})  and (\ref{eq:cauchy3})   we get 

$$   \lim_{n \rightarrow +\infty } \E || X(t+n\omega)-\tilde{X}(t)||^{2} = 0.      $$
The proof is completed.
\end{proof}

\begin{definition}
A  continuous and bounded process $X: \R_{+} \rightarrow  \LL^{2}(\PP, \mat{H})$ is said to be square
mean asymptotically $\omega$-periodic if $ X= Y+Z$ where $Y$ and $Z$ are continuous bounded processes 
such that 
$$ \E ||Y(t+\omega)- Y(t)||^{2} = 0 \quad \textrm{and} \quad \lim_{t \to \infty} \E ||Z(t)||^{2}=0.   $$
\end{definition} 
We write $Y \in P_{\omega}\big(\R_{+},\LL^{2}(\PP, \mat{H})\big) $, $ Z \in C_{0}\big(\R_{+},\LL^{2}(\PP, \mat{H})\big) $ 
and we denote the space of all square mean asymptotically $\omega$-periodic stochastic  process 
 $X: \R_{+} \rightarrow  \LL^{2}(\PP, \mat{H})$ by $ AP_{\omega}\big(\R_{+},\LL^{2}(\PP, \mat{H})\big)$. 

\begin{lemma}
The space $AP_{\omega}\big(\R_{+},\LL^{2}(\PP, \mat{H})\big) $ is a closed subspace of \\
$P_{\omega}L\big(\R_{+},\LL^{2}(\PP, \mat{H})\big) $ 
\end{lemma}
\begin{proof}
Note that   $$AP_{\omega}\big(\R_{+},\LL^{2}(\PP, \mat{H})\big) \subseteq  P_{\omega}L\big(\R_{+},\LL^{2}(\PP, \mat{H})\big).$$
Indeed, if  $X\in AP_{\omega}\big(\R_{+},\LL^{2}(\PP, \mat{H})\big)  $ then we have  for all $n\geq 1$,  $t\geq 0$, 
\begin{align*}
& \E \left|\left| X(t+n\omega)- Y(t)\right|\right|^{2} \\ 
& \leq  2\E \left|\left| X(t+n\omega)- Y(t+n\omega)\right|\right|^{2} +
  2\E \left|\left| Y(t+n\omega)- Y(t)\right|\right|^{2}\\& \leq  2\E \left|\left| Z(t+n\omega)\right|\right|^{2} +  2\E \left|\left| Y(t+n\omega)- Y(t)\right|\right|^{2}\\
  & =  2\E \left|\left| Z(t+n\omega)\right|\right|^{2}\\
  \end{align*}
  so that $$ \lim_{n\to \infty} \E \left|\left| X(t+n\omega)- Y(t)\right|\right|^{2}=0$$ for all $t\geq 0$.\\
 
 Now let's  show that  $AP_{\omega}\big(\R_{+},\LL^{2}(\PP, \mat{H})\big) $ is a closed space. \\
  Let $X \in  \overline{AP_{\omega}\big(\R_{+},\LL^{2}(\PP, \mat{H})\big)}$; there exist $X_{n} \in  AP_{\omega}\big(\R_{+},\LL^{2}(\PP, \mat{H})\big) $ such that $\lim_{n\to \infty} X_{n} = X $. Since $X_{n} \in  AP_{\omega}\big(\R_{+},\LL^{2}(\PP, \mat{H})\big) $,
  we have $X_{n} = Y_{n} + Z_{n}$ where $Y_{n} \in P_{\omega}\big(\R_{+},\LL^{2}(\PP, \mat{H})\big) $ and $ Z_{n} \in C_{0}\big(\R_{+},\LL^{2}(\PP, \mat{H})\big) $.\\
  If $Y_n$ or $Z_n$ does not converge then $X_n$ will not converge. Thus, there exist $Y$ and $Z$ such that 
   $\lim_{n\to \infty} Y_{n} = Y $ and  $\lim_{n\to \infty} Z_{n} = Z $. We have $X=Y+Z$. \\
  In the sequel we'll show that $Y \in  P_{\omega}\big(\R_{+},\LL^{2}(\PP, \mat{H})\big) $ and $ Z \in C_{0}\big(\R_{+},\LL^{2}(\PP, \mat{H})\big) $.  
  Firstly, we have 
  \begin{align*}
 & \E ||Y(t+\omega)- Y(t)||^{2} \\  &  \leq  3\E ||Y(t+\omega)- Y_{n}(t+\omega)||^{2} + 3 \E ||Y_{n}(t+\omega)- Y_{n}(t)||^{2}  \\
 & \quad    + 3\E ||Y_{n}(t)- Y(t)||^{2}\\
  & = 3\E ||Y(t+\omega)- Y_{n}(t+\omega)||^{2}   + 3\E ||Y_{n}(t)- Y(t)||^{2}\\
 \end{align*}    
 For $N$ sufficiently large, if $n \geq  N$ then 
 $$  \E ||Y(t+\omega)- Y_{n}(t+\omega)||^{2}  \leq \epsilon/6$$
 $$ \E ||Y_{n}(t)- Y(t)||^{2} \leq \epsilon/6.$$
 Therefore for $n>N$, 
 $$ E ||Y(t+\omega)- Y(t)||^{2} \leq \epsilon.$$
 Thus $$ E ||Y(t+\omega)- Y(t)||^{2} = 0,$$
 so that $Y \in  P_{\omega}\big(\R_{+},\LL^{2}(\PP, \mat{H})\big). $\\
 
 On the other hand 

   \begin{align*}
 \E ||Z(t)||^{2} &  \leq  2\E ||Z(t)- Z_{n}(t)||^{2} + 2 \E ||Z_{n}(t)||^{2}. \\
 \end{align*}  
 For all $\epsilon >0$, $\exists T_{\epsilon}>0$, $t > T_{\epsilon}$  $\Rightarrow $  
 $$ \E ||Z_{n}(t)||^{2} \leq \epsilon/4.$$
 There exists $N \in \mathbb{N}$, $n >N$ $\Rightarrow$ 
 $$E ||Z(t)- Z_{n}(t)||^{2} \leq \epsilon/4.$$
 Therefore for all  $n>N$ and $t > T_{\epsilon}$  we have 
   \begin{align*}
 \E ||Z(t)||^{2} &  \leq  \epsilon/2 +  \epsilon/2 = \epsilon,  \\
 \end{align*}  
 so that $ Z \in C_{0}\big(\R_{+},\LL^{2}(\PP, \mat{H})\big).$
\end{proof}

From the above Lemma, we have the following conclusion by the fundamental knowledge of functional analysis.

\begin{theorem}
The space $AP_{\omega}\big(\R_{+},\LL^{2}(\PP, \mat{H})\big) $ is a
 Banach space equipped with the  norm $|| . ||_{\infty} $.
\end{theorem}

The following result provides some interesting properties.

\begin{theorem}
\label{criteria}
Let $X$ be a  continuous and bounded stochastic process and $\omega >0$. Then the following statements are equivalent 
\begin{itemize}
\item[(i)] $ X \in AP_{\omega}\big(\R_{+},\LL^{2}(\PP, \mat{H})\big) $
\item[(ii)]  We have $$ \lim_{n\to \infty} \E \left|\left| X(t+n\omega)- Y(t)\right|\right|^{2}=0$$
uniformly on $t \in \R_{+}$ for some stochastic process $Y: \R_+  \rightarrow \LL^{2}(\PP, \mat{H})$.
\item[(iii)]  We have  $$ \lim_{n\to \infty} \E \left|\left| X(t+n\omega)- Y(t)\right|\right|^{2}=0$$
uniformly on compact  subsets of $\R_+$ for some stochastic process\\
 $Y: \R_+  \rightarrow \LL^{2}(\PP, \mat{H})$.
\item[(iv)] We also have 
$$ \lim_{n\to \infty} \E \left|\left| X(t+n\omega)- Y(t)\right|\right|^{2}=0$$
uniformly on $[0,\omega]$ for some stochastic process $Y: \R_+  \rightarrow \LL^{2}(\PP, \mat{H})$.
\end{itemize}
\end{theorem}

\begin{proof}
Clearly Statement (ii)  implies (iii) and (iii) implies (iv). Suppose that (i) holds and let $X=Y+Z$ where 
$Y \in P_{\omega}\big(\R_{+},\LL^{2}(\PP, \mat{H})\big) $ and $ Z \in C_{0}\big(\R_{+},\LL^{2}(\PP, \mat{H})\big) $.
Now for $n \in \N$,
\begin{equation}
\label{eq1}
 X(t+n\omega) = Y(t+n\omega) + Z(t+n\omega) \quad (\star).
\end{equation}

Let $\epsilon >0$. Since $ Z \in C_{0}\big(\R_{+},\LL^{2}(\PP, \mat{H})\big) $ there exist $N_{1}$ such that \\
$\E||Z(t+n\omega)||^{2} < \epsilon/2$ whenever $n \geq N_{1}$ for every $t\in \R_+$. Then using (\ref{eq1}), we obtain 

\begin{align*}
\E|| X(t+n\omega)-Y(t)||^{2} 
 & \leq 2 \E|| X(t+n\omega)-Y(t+n\omega)||^{2} \\
 & +2 \E|| Y(t+n\omega)-Y(t)||^{2}  \\
 &  =  2  \E||Z(t+n\omega)||^{2} \\
 &\leq \epsilon,
\end{align*}
whenever $n \geq N_{1}$ for every $t\in \R_+$. This   shows that 
 $$ \lim_{n\to \infty} \E \left|\left| X(t+n\omega)- Y(t)\right|\right|^{2}=0$$
uniformly on $t \geq 0$ for some stochastic process $Y: \R_+  \rightarrow \LL^{2}(\PP, \mat{H})$.\\
Hence (i) implies (ii).\\

Finally, suppose that (iv) holds. It is clear that $Y$ is bounded on $\R_+$ and 
 $$\E \left|\left| Y(t+\omega)- Y(t)\right|\right|^{2}=0$$ for each $t\geq 0$ like in Proposition \ref{proposition} , part (c).\\

Thus, to show the continuity of $Y$ on $\R_+$ we only need to prove that $Y$ is continuous on $[0,\omega]$. 
Now, take any fixed $t_{0} \in [0,\omega]$ and let $t\in [0,\omega]$. For each $n \in \N$, we have 
 
\begin{align}
\label{eq2}
\E \left|\left| Y(t)- Y(t_{0})\right|\right|^{2}
 & \leq 3\E \left|\left| Y(t)- X(t+n\omega)\right|\right|^{2} \\
 & + 3  \E \left|\left| X(t+n\omega)- X(t_{0}+n\omega)\right|\right|^{2}\\
 &  + 3 \E \left|\left| X(t_{0}+n\omega)-Y(t_{0})\right|\right|^{2} 
\end{align}

Let $\epsilon >0$. By Assumption in (iv), we conclude that there exists a positive integer $N_2$ such that \begin{equation}
\label{eq:1}
\E \left|\left| Y(t)- X(t+n\omega) \right|\right|^{2} < \epsilon/9
\end{equation}
for $t \in [0,\omega]$ and $n\geq N_{2}$.\\
On the other hand, since $X(t+N_{2}\omega)$ is in $C_{b}\big(\R_{+},\LL^{2}(\PP, \mat{H})\big) $ then there exists $\delta >0$ such that 

 \begin{equation}
\label{eq:2}
\E \left|\left| X(t+N_{2}\omega)- X(t_{0}+ N_{2}\omega) \right|\right|^{2} < \epsilon/9
\end{equation}
for $|t-t_{0}| < \delta$.\\
Using (\ref{eq2}) to  (\ref{eq:2}) we conclude that 
 $$ \E \left|\left| Y(t)- Y(t_{0}) \right|\right|^{2} < \epsilon \; \textrm{when} \; |t-t_{0}| < \delta,$$ 
 which show that $Y$ is continuous on $[0,\omega]$. Hence $Y \in P_{\omega}\big(\R_{+},\LL^{2}(\PP, \mat{H})\big) $. \\

Next, we will show that $X-Y \in C_{0}\big(\R_{+},\LL^{2}(\PP, \mat{H})\big) $. Suppose that $\epsilon >0$ and there exists a positive integer $N_3$  such that
 $$ \E \left|\left| X(t+n\omega)- Y(t) \right|\right|^{2} < \epsilon$$
  when $n \geq N_3$ uniformly for $t\in [0,\omega]$ by Assumption in (iv) again.\\
Thus, for $n=N_{3}+k$, $k=0,1,2,...,$ we conclude that 
 $$ \E \left|\left| X(t+(N_{3}+k)\omega)- Y(t) \right|\right|^{2} < \epsilon$$
uniformly for $t\in [0,\omega]$. Moreover, if we denote $t'=t+k\omega$, where $t'\in |k\omega, (k+1)\omega]$, $t \in [0,\omega]$ and $k=0,1,2,...,$ then we obtain 

\begin{align*}
&\E \left|\left| X(t'+N_{3}\omega)- Y(t'+N_{3}\omega)\right|\right|^{2}\\ 
 & = \E \left|\left| X(t+ (N_{3}+k)\omega)- Y(t+ (N_{3}+k)\omega)\right|\right|^{2} \\
 & = \E \left|\left| X(t+ (N_{3}+k)\omega)- Y(t)\right|\right|^{2} < \epsilon\\
\end{align*}
for $t' \in [k\omega, (k+1)\omega]$, $k=0,1,2,...$.\\
 That is  $$ \E \left|\left| X(t)- Y(t) \right|\right|^{2} < \epsilon \quad \quad (t\geq N_{3}\omega)$$
which show that $ X-Y \in C_{0}\big(\R_{+},\LL^{2}(\PP, \mat{H})\big) $. Hence $X \in \in AP_{\omega}\big(\R_{+},\LL^{2}(\PP, \mat{H})\big) $\\  and (iv) implies (i).
\end{proof}
The following generalizes the Definition 2.3.\\
\begin{definition}
A continuous bounded  process $f : \R_{+} \times \LL^{2}(\PP, \mat{H}) \rightarrow  \LL^{2}(\PP, \mat{H})$ is called square mean $\omega$ periodic limit in 
$t \in \R_{+}$ uniformly for $X$ in bounded sets of  $\LL^{2}(\PP, \mat{H})$ if for every bounded subsets $K$ of $\LL^{2}(\PP, \mat{H})$, 
$ \{ f(t,X) : t \in \R_{+}; X \in K    \} $ is bounded and 
$$  \lim_{n\to  +\infty } \E \left|\left| f(t+n\omega,X)- \tilde{f}(t,X) \right|\right|^{2} = 0     $$
is well defined when $n \in \mathbb{N}$ for each $t \geq 0$ and for some process $ \tilde{f}  : \R_{+} \times \LL^{2}(\PP, \mat{H}) \rightarrow  \LL^{2}(\PP, \mat{H})$ .
\end{definition}

We have the following composition result:

\begin{theorem}
\label{composition}
Assume that $f: \R_{+} \times \LL^{2}(\PP, \mat{H}) \rightarrow  \LL^{2}(\PP, \mat{H})$ is  a square mean  $\omega$-periodic limit process 
uniformly for $Y \in \LL^{2}(\PP, \mat{H})$  in  bounded sets of $\LL^{2}(\PP, \mat{H})$ and satisfies the Lipschitz condition, that is, there exists  constant  $L > 0$  such that $$   \E ||f(t,Y)-f(t,Z)||^{2} \leq L\, \E||Y-Z||^{2} \quad \forall t \geq 0, \, \forall  \, Y,Z \in \LL^{2}(\PP, \mat{H}). $$
 Let $X :  \R_{+}  \rightarrow  \LL^{2}(\PP, \mat{H})$ be a square mean  $\omega$-periodic limit process.  Then the process $ F(t) = (f(t,X(t)))_{t\geq 0}$ is  a square mean  $\omega$-periodic limit process. 
\end{theorem}

\begin{proof}
Since $X$ is a square mean $\omega$-periodic limit process we have : 
\begin{equation}
\label{eq:ref1}
   \lim_{n\to  +\infty } \E \left|\left| X(t+n\omega)- \tilde{X}(t)  \right|\right|^{2} = 0    
   \end{equation}
for some stochastic process $\tilde{X} :  \R_{+}  \rightarrow  \LL^{2}(\PP, \mat{H})$.\\

By using  Proposition 2.1 (4), we can choose a bounded subset $K$ of  $\LL^{2}(\PP, \mat{H})$ such that $ X(t), \tilde{X}(t)  \in K$ for all $t\geq 0$. Then 
$F(t)$ is bounded.\\
On the other hand we have :  
\begin{equation}
\label{eq:ref2}
   \lim_{n\to  +\infty } \E \left|\left| f(t+n\omega,X)-  \tilde{f}(t,X)  \right|\right|^{2} = 0, 
   \end{equation}
for each $t\geq 0$ and each $X\in K$.\\
Let us consider the process $\tilde{F} : \R_{+}  \rightarrow  \LL^{2}(\PP, \mat{H})$ defined by $\tilde{F} (t) = \tilde{f}(t,\tilde{X}(t))$.\\ Note that 
\begin{align*}
\E \left|\left| F(t+n\omega)-  \tilde{F} (t)  \right|\right|^{2}
 & = \E \left|\left| f(t+n\omega, \tilde{X}(t+n\omega))-  \tilde{f}(t,\tilde{X}(t))  \right|\right|^{2}\\
 &  \leq  2 \E \left|\left| f(t+n\omega,X(t+n\omega))- f(t+n\omega,\tilde{X}(t))   \right|\right|^{2} \\
 &  + 2 \E \left|\left| f(t+n\omega,\tilde{X}(t))- \tilde{f}(t,\tilde{X}(t))   \right|\right|^{2} \\
 & \leq 2L \, \E \left|\left| X(t+n\omega)- \tilde{X}(t)   \right|\right|^{2} \\
 & + 2  \E \left|\left| f(t+n\omega,\tilde{X}(t))- \tilde{f}(t,\tilde{X}(t))   \right|\right|^{2} \\
\end{align*}

We deduce from (\ref{eq:ref1}) and (\ref{eq:ref2} ) that 
$$ \lim_{n\to  +\infty }   \E \left|\left| F(t+n\omega)-  \tilde{F} (t)  \right|\right|^{2} = 0 $$
is well defined for $\tilde{F}(t) = \tilde{f}(t,\tilde{X}(t))$. 

\end{proof}

Now, let us end this section with the following property of a Brownian motion. 

\begin{pro}[Weak Markov Property]
Let $B = (B(s))_{s\geq 0} $ be a two-sided  standard one-dimensional Brownian motion  and set for $h \in \mathbb{R}$, $$ \tilde{B}^{h}(u) = B(u+h) - B(h), u \in \mathbb{R}.$$
 Then the process $\tilde{B}^{h}$ is a two-sided  Brownian motion indépendent of $ \{ B(s) : s \leq h  \}$.  In others words, 
 $B(u+h)$  has the same law as $\tilde{B}^{h}(u) + B(h)$. 

\end{pro}

\section{A Stochastic Differential Equation}
In this section, we investigate the existence of the square mean  asymptotically $\omega$-periodic solution
 to the following SDE : 

\begin{equation}
\label{eq: base}
\left\{
\begin{array}{l}
dX(t) =  A X(t)dt + f(t,X(t))dt  + g(t,X(t))dB(t), \quad \quad t \geq 0 \\
X(0) =  c_{0},
\end{array}
\right.
\end{equation}

where $A$ is a  closed linear operator  and 
$$   f: \R_{+} \times \LL^{2}(\PP, \mat{H}) \rightarrow  \LL^{2}(\PP, \mat{H}),    $$

$$  g: \R_{+} \times \LL^{2}(\PP, \mat{H}) \rightarrow  \LL^{2}(\PP, \mat{H})  $$
are   Lipschitz continuous and  bounded, $(B(t))_{t}$ is a two-sided standard one-dimensional Brownian motion 
 with values in $\mat{H}$ and $\cal{F}_t$-adapted and $ c_{0} \in \LL^{2}(\PP, \mat{H}) $.  Recall that $\cal{F}_{t}=\sigma\{ B(u)-B(v) / u, v \leq t \}$.\\

In order to establish our main result, we impose the following conditions.\\

{ \bf (H1)}: $A$ generates an exponentially  stable semigroup  $(T(t))_{t\geq 0}$ in $\LL^{2}(\PP, \mat{H})$, that is,  a linear operator such that  :\\

\begin{itemize}
\item[1.] $T(0)=I$  where $I$ is the identity operator.
\item[2.] $T(t)T(r)=T(t+r)$ for all $t , r \geq 0.$
\item[3.] The map $ t \mapsto T(t)x$ is continuous for every fixed $x \in \LL^{2}(\PP, \mat{H})$.
\item[4.] There exist $M>0$ and $a>0$ such that $||T(t)|| \leq Me^{-at}$ for $t\geq 0.$\\
\end{itemize}

\begin{definition} 
The $\FF_{t}$-progressively measurable process $\{X(t),\; t \geq 0 \} $ is said to be a mild solution of (\ref{eq: base})  if  it satisfies the following stochastic integral equation: 
\begin{align*}
X(t) & = T(t) c_{0} + \int_{0}^{t} T(t-s)f(s,X(s))ds +  \int_{0}^{t} T(t-s)g(s,X(s))dB(s).\\  
\end{align*}
\end{definition}
Now, we'll establish some technical results. 
\begin{lemma}
\label{lemme1}
Let $F$ be a square mean $\omega$-periodic limit process in $  \LL^{2}(\PP, \mat{H})$. 
Under Assumption  (H1)  the sequence of stochastic processes $(X_{n}(t))_{n\geq 1}$, $t\geq 0$, defined by 
$$ X_{n}(t) =   \int_{0}^{n\omega} T(t+s) F(n\omega -s)ds $$
 is a Cauchy sequence in $  \LL^{2}(\PP, \mat{H})$ for all  $t \geq 0$.
\end{lemma}
We  shall denote by $ U=(U(t))_{t\geq 0}$ the limit process  of $(X_{n}(t))_{n\geq 1}$, $t \geq 0$,  in $  \LL^{2}(\PP, \mat{H})$.

\begin{proof}
We have, 
\begin{align*}
& \E || X_{n+p}(t)-X_{n}(t)||^{2} \\
& = \E \left|\left|  \int_{0}^{(n+p)\omega} T(t+s) F((n+p)\omega -s)ds  - \int_{0}^{n\omega} T(t+s) F(n\omega -s)ds    \right|\right|^{2} \\
 & \leq 2  \E \left|\left|  \int_{n\omega}^{(n+p)\omega}T(t+s) F((n+p)\omega -s)ds  \right|\right|^{2} \\
 & +   2  \E \left|\left|  \int_{0}^{n\omega} T(t+s) \big( F((n+p)\omega -s)- F(n\omega -s) \big) ds  \right|\right|^{2} \\
 & =  I_{1}(t,n,p) + I_{2}(t,n,p)\\
\end{align*}
where 
$$   I_{1}(t,n,p)  =  2  \E \left|\left|  \int_{n\omega}^{(n+p)\omega}T(t+s) F((n+p)\omega -s)ds  \right|\right|^{2}    $$
$$  I_{2}(t,n,p) =   2  \E \left|\left|  \int_{0}^{n\omega} T(t+s) \big( F((n+p)\omega -s)- F(n\omega -s) \big) ds  \right|\right|^{2}.  $$

We have 
\begin{align*}
I_{1}(t,n,p) = 2 \E \left|\left|  \int_{n\omega}^{(n+p)\omega} T(t+s) F((n+p)\omega -s)ds \right|\right|^{2} 
\end{align*}

\begin{align*}
  & \leq 2 \E  \Big( \int_{n\omega}^{(n+p)\omega} \left|\left| T(t+s) F((n+p)\omega -s) \right|\right|ds\Big)^{2}\\
 & \leq 2  \E \Big( \int_{n\omega}^{(n+p)\omega} || T(t+s)|| \; || F((n+p)\omega -s)|| ds\Big)^{2}\\
 & \leq  2  \E \Big( \int_{n\omega}^{(n+p)\omega} Me^{-a(t+s)}|| F((n+p)\omega -s)|| ds\Big)^{2}\\
& \leq 2  \int_{n\omega}^{(n+p)\omega} M^{2}e^{-2a(t+s)}ds \int_{n\omega}^{(n+p)\omega} \E|| F((n+p)\omega -s)||^{2} ds\\
& \leq 2M^{2}p \omega K \int_{n\omega}^{+\infty} e^{-2as} ds\\
& \leq \frac{2M^{2}p \omega K}{2a} e^{-2a n \omega}.
 \end{align*}
 
Now we consider the integers  $N_1$ and $N_2$ such that
$$ \frac{M^{2}p\omega K}{a}e^{-2an\omega}  \leq  \epsilon  \quad \forall n\geq N_1$$

$$   \frac{8M^{2}K}{a^{2}} e^{-2an\omega} < \epsilon  \quad \forall n\geq N_2$$

and set  $ N= \max{ (N_{1},N_{2})}$. For $n\geq N$, we have : 

\begin{align*}
I_{2}(t,n,p) = 2 \E \left|\left|  \int_{0}^{n\omega} T(t+s)\big( F((n+p)\omega -s)-F(n\omega -s) \big)ds \right|\right|^{2} 
\end{align*}
\begin{align*}
  & \leq 2 \E \Big(\int_{0}^{n\omega} || T(t+s)||\,|| F((n+p)\omega -s)-F(n\omega -s)|| ds \Big)^{2}\\
 & \leq 4 \E \Big(\int_{0}^{N\omega} || T(t+s)||\,|| F((n+p)\omega -s)-F(n\omega -s)|| ds \Big)^{2}\\
 & + 4 \E \Big(\int_{N\omega}^{n\omega} || T(t+s)||\,|| F((n+p)\omega -s)-F(n\omega -s)|| ds \Big)^{2}\\
 & \leq I_{3}(t,n,p) + I_{4}(t,n,p) \\
\end{align*}
where 
$$   I_{3}(t,n,p) = 4 \E \Big(\int_{0}^{N\omega} || T(t+s)||\,|| F((n+p)\omega -s)-F(n\omega -s)|| ds \Big)^{2}    $$
$$  I_{4}(t,n,p)  = 4 \E \Big(\int_{N\omega}^{n\omega} || T(t+s)||\,|| F((n+p)\omega -s)-F(n\omega -s)|| ds \Big)^{2}.    $$

From $F \in P_{\omega}L\big(\R_{+},\LL^{2}(\PP, \mat{H})\big)$ , there exist a  stochastic process\\
  $\tilde{F} : \R_{+} \rightarrow  \LL^{2}(\PP, \mat{H})$ such that 
 $$\lim_{ k \rightarrow +\infty } \E || F(t+k\omega)-\tilde{F}(t)||^{2} = 0$$
 is well defined for each $t\geq 0$ when $k \in \mathbb{N}$.\\

Now, wee have :
\begin{align*}
  & I_{3}(t,n,p) \\ 
  &  \leq  8  \E \Big(\int_{0}^{N\omega} Me^{-a(t-s+N\omega)}\,|| F((n-N+p)\omega +s)-\tilde{F}(s)|| ds \Big)^{2}\\
 & +  8  \E \Big(\int_{0}^{N\omega} Me^{-a(t-s+N\omega)}\,|| F((n-N)\omega +s)-\tilde{F} (s)|| ds \Big)^{2},\\
\end{align*}

and by Cauchy Schwarz inequality, we obtain 
\begin{align*}
  &  I_{3}(t,n,p) \\
  & \leq 8 \int_{0}^{N\omega} M^{2} e^{-2a(t-s+N\omega)}ds   \int_{0}^{N\omega} \E || F((n-N+p)\omega +s)-\tilde{F} (s)||^{2} ds\\
 & + 8   \int_{0}^{N\omega} M^{2} e^{-2a(t-s+N\omega)}ds  \int_{0}^{N\omega} \E || F((n-N)\omega +s)-\tilde{F} (s)||^{2} ds\\
 & \leq 8M^{2}  \int_{0}^{N\omega} e^{-2a(-s+N\omega)}ds   \int_{0}^{N\omega} \E || F((n-N+p)\omega +s)-\tilde{F} (s)||^{2} ds\\
 & + 8 M^{2}  \int_{0}^{N\omega}  e^{-2a(-s+N\omega)}ds \Big)   \int_{0}^{N\omega} \E || F((n-N)\omega +s)-\tilde{F} (s)||^{2} ds\\
 & \leq  \frac{8M^{2}}{2a} \int_{0}^{N\omega} \E || F((n-N+p)\omega +s)-\tilde{F} (s)||^{2} ds\\
 & + \frac{8M^{2}}{2a} \int_{0}^{N\omega} \E || F((n-N)\omega +s)-\tilde{F} (s)||^{2} ds\\
\end{align*}

Using the fact that   $$  \max \{  \E|| F((n-N+p)\omega + s) - \tilde{F} (s)||^{2}, \E|| F((n-N)\omega + s) -\tilde{F} (s)||^{2} \} \leq 2K ,$$
it follows that 
$$ \lim_{n\to \infty}\, I_{3}(t,n,p)= 0  \quad \textrm{for all }  \; t \geq 0,$$ 
by Lebesgue's dominated convergence theorem. 

Similarly,
\begin{equation*}
I_{4}(t,n,p) = 4 \E \left|\left|  \int_{N\omega}^{n\omega} T(t+s)\big( F((n+p)\omega -s)-F(n\omega -s) \big)ds \right|\right|^{2}
\end{equation*}

\begin{align*}
  & \leq 4 \E \Big(\int_{N\omega}^{n\omega}  || T(t+s)||\,|| F((n+p)\omega -s)-F(n\omega -s)|| ds \Big)^{2}\\
 & \leq 4 \E \Big(\int_{N\omega}^{n\omega} Me^{-a(t+s)} \,|| F((n+p)\omega -s)-F(n\omega -s)|| ds \Big)^{2}\\
 &=  4 \E \Big(\int_{N\omega}^{n\omega} Me^{\frac{-a}{2} (t+s)} \, e^{\frac{-a}{2} (t+s)} \,|| F((n+p)\omega -s)-F(n\omega -s)|| ds \Big)^{2}.\\
\end{align*}
Again, using Cauchy Schwarz inequality, it follows that 

\begin{align*}
I_{4} & \leq   4 \int_{N\omega}^{n\omega} M^{2}e^{-a(t+s)} ds  \int_{N\omega}^{n\omega}  e^{-a(t+s)} \E || F((n+p)\omega -s)-F(n\omega -s)||^{2} ds \\
& \leq  4 \int_{N\omega}^{+\infty} M^{2}e^{-as} ds  \int_{N\omega}^{+\infty}  e^{-as} \E || F((n+p)\omega -s)-F(n\omega -s)||^{2} ds. \\
\end{align*}

Since $\E || F((n+p)\omega -s)-F(n\omega -s)||^{2} \leq 2K$, we obtain 
$$ I_{4} \leq \frac{8M^{2}K}{a^{2}}e^{-2a N \omega}  $$
and hence
$I_{4} \leq \epsilon$.\\
This shows that $(X_{n}(t))_{n\geq 1}$, $t\geq 0$,  is a Cauchy sequence in $\LL^{2}(\PP, \mat{H})$.
\end{proof}

\begin{lemma}
\label{lemme2}
Let $F$ be a square mean $\omega$-periodic limit  process in $  \LL^{2}(\PP, \mat{H})$ \\
such that 
$$ \lim_{n \to +\infty} \E || F(t+n\omega) -  \tilde{F}(t) ||^{2}  = 0$$  for all $t \geq 0$. 
Define $  V(t) = \int_{0}^{t} T(t-s)F(s)ds.$
 Under Assumption (H1) \\ we have   $$ \lim_{n \to +\infty} \E || V(t+n\omega) - V^{*}(t) ||^{2}  = 0$$  
 uniformly on $t \geq 0$ where 
 $$ V^{*}(t)  = U(t) +  \int_{0}^{t} T(t-s)\tilde{F}(s)ds.$$
\end{lemma}

\begin{proof}
Let us rewrite 
\begin{align*}
V(t+n\omega) 
 & = \int_{0}^{t+n\omega} T(t+n\omega -s) F(s)ds \\
 & = \int_{-n\omega}^{t} T(t-s) F(s+n\omega)ds \\
 &=  \int_{-n\omega}^{0} T(t-s) F(s+n\omega)ds  +  \int_{0}^{t} T(t-s) F(s+n\omega)ds \\
 & =  \int_{0}^{n\omega} T(t+s) F(s+n\omega)ds + \int_{0}^{t}  T(t-s) F(s+n\omega)ds  \\
 & = X_{n}(t,n) + I(t,n). \\
\end{align*}
We have, 
 \begin{align*}
\E || V(t+n\omega) - V^{*}(t) ||^{2} 
 & = \E \left| \left|  X_{n}(t)+ I(t,n) - U(t) -  \int_{0}^{t} T(t-s)\tilde{F}(s)ds  \right| \right|^{2}\\
 & \leq 2\E \left| \left|  X_{n}(t)- U(t) \right| \right|^{2} \\
 & + 2\E \left| \left|  I(t,n)  -  \int_{0}^{t} T(t-s)\tilde{F}(s)ds  \right| \right|^{2}.
\end{align*}

Using  Lemma \ref{lemme1}, it follows  that $$ \E \left| \left|  X_{n}(t)- U(t) \right| \right|^{2}  \rightarrow 0$$
for all $t \geq 0$. \\

Note that  for $m \omega \leq t < (m+1)\omega $;  $m\in \N$, one has

\begin{align*}
& \E \Big|\Big| I(t,n)-\int_{0}^{t} T(t-s)\tilde{F}(s)ds\Big|\Big|^{2}\\ 
 &= \E  \Big|\Big| \int_{0}^{t}T(t-s)\big(F(s+n\omega)-\tilde{F}(s)\big)ds \Big|\Big|^{2}\\
 &\leq   \E   \Big(\int_{0}^{t}||T(t-s)||\,||F(s+n\omega)-\tilde{F}(s)||ds \Big)^{2}\\
 & \leq   \E   \Big(\int_{0}^{t} Me^{-a(t-s)}\,||F(s+n\omega)-\tilde{F}(s)||ds \Big)^{2}\\
 & \leq   2 \E   \Big(\int_{0}^{m\omega} Me^{-a(t-s)}\,||F(s+n\omega)-\tilde{F}(s)||ds \Big)^{2}\\
 & + 2  \E   \Big(\int_{m\omega}^{t} Me^{-a(t-s)}\,||F(s+n\omega)-\tilde{F}(s)||ds \Big)^{2}.\\
 \end{align*}

But firstly,   $$ 2 \E   \Big(\int_{0}^{m\omega} Me^{-a(t-s)}\,||F(s+n\omega)-\tilde{F}(s)||ds \Big)^{2}$$
\begin{align*}
   & \leq 2 M^{2} \E \Big( \sum_{k=0}^{m-1} \int_{k\omega}^{(k+1)\omega} e^{-a(t-(k+1)\omega )}\,||F(s+n\omega)-\tilde{F}(s)||ds \Big)^{2}\\
    &=   2 M^{2} \E \Big( \sum_{k=0}^{m-1} \int_{0}^{\omega} e^{-a(t-(k+1)\omega )}\,||F(s+(n+k)\omega)-\tilde{F}(s+k\omega)||ds \Big)^{2}\\
    & \leq  2 M^{2} \int_{0}^{\omega} \big( \sum_{k=0}^{m-1}  e^{-a(t-(k+1)\omega )}\big)^{2}ds \int_{0}^{\omega} \E ||F(s+(n+k)\omega)-\tilde{F}(s+k\omega)||^{2}ds.\\
\end{align*}
Since $$\lim_{n\to +\infty} \E ||F(s+(n+k)\omega)-\tilde{F}(s+k\omega)||^{2} = 0 $$ for $s \in [0,\omega]$ and the fact that 
$$ \E ||F(s+(n+k)\omega)-\tilde{F}(s+k\omega)||^{2} \leq 2 K,$$ \\
it follows by Lebesgue's dominated convergence theorem that : 

$$ \lim_{n\to +\infty} \int_{0}^{\omega}  \E ||F(s+(n+k)\omega)-\tilde{F}(s+k\omega)||^{2} = 0. $$

Note that 
$$   \sum_{k=0}^{m-1}  e^{-a(t-(k+1)\omega ) }\leq  \frac{1}{1-e^{-a \omega}}   \quad \textrm{unifmorly in } t,  $$
therefore 
$$ 2 \E   \Big(\int_{0}^{m\omega} Me^{-a(t-s)}\,||F(s+n\omega)-\tilde{F}(s)||ds \Big)^{2} \leq \frac{2M^{2}\omega}{(1-e^{-a\omega})^{2}} \epsilon. $$

On the other hand,  

$$
2  \E  \Big(\int_{m\omega}^{t} Me^{-a(t-s)}\,||F(s+n\omega)-\tilde{F}(s)||ds \Big)^{2} $$

$$ \leq  2M^{2}  \E   \Big(\int_{m\omega}^{(m+1)\omega} ||F(s+n\omega)-\tilde{F}(s)||ds \Big)^{2}$$

$$ =  2M^{2}   \E   \Big(\int_{0}^{\omega} \,||F(s+(n+m)\omega)-\tilde{F}(s+m\omega)||ds \Big)^{2}$$

$$ \leq 2 M^{2} \omega  \int_{0}^{\omega} \E||F(s+(n+m)\omega)-\tilde{F}(s+m\omega)||^{2}ds.$$

But, 
\begin{align*}
\E||F(s+(n+m)\omega)-\tilde{F}(s+m\omega)||^{2} & \leq 2 \E||F(s+(n+m)\omega)-\tilde{F}(s)||^{2} \\
& + 2 \E||\tilde{F}(s+m\omega)-\tilde{F}(s)||^{2}\\
\end{align*}

so that $$ \lim_{n \to +\infty}  \E||F(s+(n+m)\omega)-\tilde{F}(s+m\omega)||^{2}  = 0.$$
Again, using the Lebesgue's dominated convergence theorem   we have 

$$      \lim_{n \to +\infty}  \int_{0}^{\omega} \E||F(s+(n+m)\omega)-\tilde{F}(s+m\omega)||^{2}ds = 0,  $$
and hence 
$$
 \lim_{n \to +\infty}   2  \E  \Big(\int_{m\omega}^{t} Me^{-a(t-s)}\,||F(s+n\omega)-\tilde{F}(s)||ds \Big)^{2}  = 0.$$
Thus 
 
 $$  \lim_{n \to +\infty}   \E \Big|\Big| I(t,n)-\int_{0}^{t} T(t-s)\tilde{F}(s)ds\Big|\Big|^{2} = 0$$
 for all $ t \geq 0$. \\ 
 
 Therefore $$ \lim_{n \to +\infty} \E || V(t+n\omega) - V^{*}(t) ||^{2}  = 0$$ 
uniformly on  $t \geq 0$   for some stochastic process $V^{*}(t) : \R_{+} \rightarrow  \LL^{2}(\PP, \mat{H})$.   
\end{proof}
 
 \begin{lemma}
 \label{lemme3}
Let $G$ be a square mean $\omega$-periodic limit process in $  \LL^{2}(\PP, \mat{H})$ and \\
$B(t)$  a two-sided standard one-dimensional Brownian motion. \\ Under Assumption (H1) the sequence of stochastic process
 $(Y_{n}(t))_{n\geq 1}$, $t\geq 0$, \\  defined by 
$$ Y_{n}(t) =  \int_{-n\omega}^{0} T(t-s) G(s+n\omega)dB(s) $$
 is a Cauchy sequence in $  \LL^{2}(\PP, \mat{H})$ for all  $t \geq 0$. 
\end{lemma}
We  will denote by $ U^{*}=(U^{*}(t))_{t\geq 0}$ the limit process of  $(Y_{n}(t))_{n\geq 1}$, $t\geq 0$,  in $  \LL^{2}(\PP, \mat{H})$.
\begin{proof}
We have,
\begin{align*}
& \E || Y_{n+p}(t)-Y_{n}(t)||^{2}  = \E \left|\left|  \int_{-(n+p)\omega}^{0} T(t-s) G((n+p)\omega +s)dB(s) \right.\right.\\ 
 & \left. \left. - \int_{-n\omega}^{0} T(t-s) G(n\omega +s)dB(s)\right|\right|^{2} \\
& \leq  2\E \left|\left|  \int_{-(n+p)\omega}^{-n\omega} T(t-s) G((n+p)\omega +s)dB(s)\right|\right|^{2} \\
&   + 2\E \left|\left|  \ \int_{-n\omega}^{0} T(t-s) \Big( G((n+p)\omega +s)- G(n\omega +s) \Big) dB(s)    \right|\right|^{2} \\
& \leq  2\E \Big(  \int_{-(n+p)\omega}^{-n\omega} ||T(t-s)|| \,|| G((n+p)\omega +s)|| dB(s)\Big)^{2} \\
&   + 2\E   \Big(   \int_{-n\omega}^{0}  ||T(t-s)|| \, || G((n+p)\omega +s)- G(n\omega +s)|| dB(s)    \Big)^{2} \\
& \leq  2\E  \int_{-(n+p)\omega}^{-n\omega} ||T(t-s)||^{2} \,|| G((n+p)\omega +s)||^{2} ds \\
&   + 2\E  \int_{-n\omega}^{0} ||T(t-s) ||^{2} ||G((n+p)\omega +s)- G(n\omega +s)||^{2} ds \\
 & \leq 2  \E  \int_{n\omega}^{(n+p)\omega} || T(t+s)||^{2} \; || G((n+p)\omega -s)||^{2}  ds \\
    & + 2  \E \int_{0}^{n\omega} || T(t+s)||^{2} \;||G((n+p)\omega -s)- G(n\omega -s)||^{2} ds \\
    & \leq J_{1}(t,n,p)  +   J_{2}(t,n,p)
   \end{align*}
where 
$$  J_{1}(t,n,p)  = 2  \E  \int_{n\omega}^{(n+p)\omega} || T(t+s)||^{2} \; || G((n+p)\omega -s)||^{2}  ds$$
$$   J_{2}(t,n,p) = 2  \E \int_{0}^{n\omega} || T(t+s)||^{2} \;||G((n+p)\omega -s)- G(n\omega -s)||^{2} ds. $$

Estimates of  $J_{1}(t,n,p)$. 
\begin{align*}
J_{1}(t,n,p) = 2  \E \int_{n\omega}^{(n+p)\omega} || T(t+s)||^{2} \; || G((n+p)\omega -s)||^{2} ds.
\end{align*}

\begin{align*}
 & \leq 2  \E \int_{n\omega}^{(n+p)\omega}|| T(t+s)||^{2} \; || G((n+p)\omega -s)||^{2} ds\\
 & \leq  2   \int_{n\omega}^{(n+p)\omega} M^{2}e^{-2a(t+s)}\E|| G((n+p)\omega -s)||^{2} ds \\
 & \leq 2 M^{2} K \int_{n\omega}^{(n+p)\omega}  e^{-2as}ds  \\
  & \leq 2 M^{2} K \int_{n\omega}^{+\infty}  e^{-2as}ds  \\
& \leq \frac{2M^{2}K}{2a} e^{-2a n \omega}
 \end{align*}
Now we consider the integers  $N_1$ and $N_2$ such that
$$ \frac{M^{2} K}{a}e^{-2an\omega}  \leq  \epsilon  \quad \forall n\geq N_1$$

$$   \frac{4M^{2}K}{a} e^{-2an\omega}  \leq \epsilon  \quad \forall n\geq N_2$$

and set  $ N= \max{ (N_{1},N_{2})}$. For $n\geq N$, we have : 

\begin{align*}
J_{2}(t,n,p) = 2  \E \int_{0}^{n\omega} || T(t+s)||^{2} \;||G((n+p)\omega -s)- G(n\omega -s)||^{2} ds
\end{align*}

\begin{align*}
   & \leq 2 \int_{0}^{n\omega} M^{2}e^{-2a(t+s)}  \; \E  ||G((n+p)\omega -s)- G(n\omega -s)||^{2} ds\\
 & \leq  2  \int_{0}^{N\omega} M^{2}e^{-2a(t+s)}\, \E  ||G((n+p)\omega -s)- G(n\omega -s)||^{2}ds \\
 & + 2  \int_{N\omega}^{n\omega} M^{2}e^{-2a(t+s)}\, \E  ||G((n+p)\omega -s)- G(n\omega -s)||^{2}ds \\
 & \leq J_{3}(t,n,p) + J_{4}(t,n,p) \\
\end{align*}
where 
$$   J_{3}(t,n,p) = 2  \int_{0}^{N\omega} M^{2}e^{-2a(t+s)}\, \E  ||G((n+p)\omega -s)- G(n\omega -s)||^{2}ds      $$
$$    J_{4}(t,n,p) =  2  \int_{N\omega}^{n\omega} M^{2}e^{-2a(t+s)}\, \E  ||G((n+p)\omega -s)- G(n\omega -s)||^{2}ds.    $$

Since $G$ is a square mean $\omega$ periodic limit process then 
 $$\lim_{ k \rightarrow +\infty } \E || G(t+k\omega)-\tilde{G}(t)||^{2} = 0, \quad \forall t \geq 0$$
 is well defined for each $t\geq 0$ when $k \in \mathbb{N}$ for some stochastic process\\
 $\tilde{G} : \R_{+} \rightarrow  \LL^{2}(\PP, \mat{H})$.

Now, 
\begin{align*}
  & J_{3}(t,n,p) =  2  \int_{0}^{N\omega} M^{2}e^{-2a(t+s)}\,\E || G((n+p)\omega -s)-G(n\omega -s)||^{2}ds \\
  & = 2  \int_{0}^{N\omega} M^{2}e^{-2a(t-s+N\omega)}\,\E || G((n-N+p)\omega + s)- G((n-N)\omega + s)||^{2}ds \\
  & \leq   4   \int_{0}^{N\omega} M^{2}e^{-2a(t-s+N\omega)}\, \E|| G((n-N+p)\omega + s)-\tilde{G} (s)||^{2}ds \\
 & +  4  \int_{0}^{N\omega} M^{2}e^{-2a(t-s+N\omega)}\, \E||G((n-N)\omega + s)-\tilde{G} (s)||^{2}ds \\
\end{align*}
Since 
$$ \lim_{n\to \infty}\, \E|| G((n-N+p)\omega + s)-\tilde{G} (s)||^{2} = 0  $$ 
$$ \lim_{n\to \infty}\, \E|| G((n-N)\omega + s)-\tilde{G} (s)||^{2} = 0  $$ 
and  using the fact that \\  
$ \max \{  \E|| G((n-N+p)\omega + s) -\tilde{G} (s)||^{2}, \E|| G((n-N)\omega + s) -\tilde{G} (s)||^{2} \} \leq 2K,$\\
it follows by  Lebesgue's dominated convergence theorem that 

$$ \lim_{n\to \infty}\, J_{3}(t,n,p)= 0  $$
for all $t \geq 0$. \\

On the other hand,
\begin{equation*}
J_{4}(t,n,p) =  2  \int_{N\omega}^{n\omega} M^{2}e^{-2a(t+s)}\, \E  ||G((n+p)\omega -s)- G(n\omega -s)||^{2}ds. 
\end{equation*}

Since $\E || G((n+p)\omega -s)-G(n\omega -s)||^{2} \leq 2K$,  we get
\begin{align*}
J_{4}(t,n,p) &  \leq 4M^{2}K   \int_{N\omega}^{+\infty} M^{2}e^{-2as} ds  \\
& \leq   \frac{4M^{2}K}{a} e^{-2aN\omega}\\
& \leq   \frac{4M^{2}K}{a} e^{-2aN_{2}\omega}\\ 
\end{align*}
so that 
$$J_{4}(t,n,p) \leq \epsilon.$$

This shows that $(Y_{n}(t))_{n\geq 1}$, $t\geq 0$, is a Cauchy sequence in $\LL^{2}(\PP, \mat{H})$ 
\end{proof}
 
\begin{lemma}
\label{lemme4}
Let $G$ be square mean $\omega$-periodic limit in $  \LL^{2}(\PP, \mat{H})$ such that 
$$ \lim_{n \to +\infty} \E || G(t+n\omega) - \tilde{G} (t) ||^{2}  = 0$$  for all $t \geq 0$. 
Define
$$  H(t) = \int_{0}^{t} T(t-s)G(s)dB(s) $$ Under Assumption (H1) we have 
  $$ \lim_{n \to +\infty} \E || H(t+n\omega) - H^{*}(t) ||^{2}  = 0$$  uniformly on  $t \geq 0$ where 
 $$ H^{*}(t)  = U^{*}(t) +  \int_{0}^{t} T(t-s)\tilde{G} (s)dB(s).$$ 
\end{lemma}

\begin{proof}
Let us rewrite, 
\begin{align*}
H(t+n\omega) 
 & = \int_{0}^{t+n\omega} T(t+n\omega -s) G(s)dB(s) \\
 & = \int_{-n\omega}^{t} T(t-s) G(s+n\omega)dB(s+n\omega). \\
 \end{align*}
 Let $\tilde{B}(s) = B(s+n\omega)-B(n\omega)$ for each $s \in \mathbb{R}$.  By the weak Markov property 
 $\tilde{B}$ is also a two sided Brownian motion and has the same distribution as $B$. Moreover 
 $\{ \tilde{B}(s), s \in \mathbb{R} \}$ is a two sided  Brownian motion independent of $B(n\omega)$.
 Thus  
 \begin{align*}
H(t+n\omega) 
 &=  \int_{-n\omega}^{t} T(t-s) G(s+n\omega)d\tilde{B}(s)\\
 &=  \int_{-n\omega}^{0} T(t-s) G(s+n\omega)d\tilde{B}(s)  +  \int_{0}^{t} T(t-s) G(s+n\omega)d\tilde{B}(s)\\
 & = Y_{n}(t) + J(t,n) \\
\end{align*}
where $$  J(t,n) =  \int_{0}^{t} T(t-s) G(s+n\omega)d\tilde{B}(s)=  \int_{0}^{t} T(t-s) G(s+n\omega)dB(s). $$
We have 
 \begin{align*}
& \E || H(t+n\omega) - H^{*}(t) ||^{2} \\ 
 & = \E \left| \left|  Y_{n}(t) + J(t,n) - U^{*}(t) -  \int_{0}^{t} T(t-s)\tilde{G} (s)ds  \right| \right|^{2}\\
 & \leq \E \left| \left|  Y_{n}(t)- U^{*}(t) \right| \right|^{2} \\
 & + \E \left| \left|  J(t,n)  -  \int_{0}^{t} T(t-s)\tilde{G} (s)ds  \right| \right|^{2}
\end{align*}

Using  Lemma \ref{lemme3}, it follows that  $$ \E \left| \left|  Y_{n}(t)- U^{*}(t) \right| \right|^{2}  \rightarrow 0$$
for all $t \geq 0$, when $ n \to +\infty$.  \\

For $m \omega \leq t < (m+1)\omega $;  $m\in \N$, one has
\begin{align*}
&\E \Big|\Big| J(t,n)-\int_{0}^{t} T(t-s)\tilde{G} (s)dB(s)  \Big|\Big|^{2}\\ 
 &= \E  \Big|\Big| \int_{0}^{t}T(t-s)\big(G(s+n\omega)-\tilde{G} (s)\big)dB(s)  \Big|\Big|^{2}\\
 &\leq   \E   \Big(\int_{0}^{t}||T(t-s)||\,||G(s+n\omega)-\tilde{G} (s)||dB(s)  \Big)^{2}\\
 & \leq   \E   \Big(\int_{0}^{t} Me^{-a(t-s)}\,||G(s+n\omega)-\tilde{G} (s)||dB(s)  \Big)^{2}\\
 & = \int_{0}^{t} M^{2}e^{-2a(t-s)}\, \E ||G(s+n\omega)-\tilde{G} (s)||^{2}ds \\
 & \leq    \int_{0}^{m\omega}  M^{2}e^{-2a(t-s)}\, \E ||G(s+n\omega)-\tilde{G} (s)||^{2}ds \\
 & +  \int_{m\omega}^{t}  M^{2}e^{-2a(t-s)}\, \E ||G(s+n\omega)-\tilde{G} (s)||^{2}ds. \\
 \end{align*}

Now,    $$  \int_{0}^{m\omega}  M^{2}e^{-2a(t-s)}\, \E ||G(s+n\omega)-\tilde{G} (s)||^{2}ds$$ 
\begin{align*}
   & \leq M^{2}  \sum_{k=0}^{m-1} \int_{k\omega}^{(k+1)\omega} e^{-2a(t-(k+1)\omega )}\,\E||G(s+n\omega)-\tilde{G} (s)||^{2} ds \\
    &=   M^{2} \int_{0}^{\omega}  \sum_{k=0}^{m-1}e^{-2a(t-(k+1)\omega )}\,\E ||G(s+(n+k)\omega)-\tilde{G} (s+k\omega)||^{2}ds\\
    & \leq 2   M^{2} \int_{0}^{\omega}  \sum_{k=0}^{m-1}e^{-2a(t-(k+1)\omega )}\,\E ||G(s+(n+k)\omega)-\tilde{G} (s)||^{2}ds\\
    & +  2 M^{2} \int_{0}^{\omega}  \sum_{k=0}^{m-1}e^{-2a(t-(k+1)\omega )}\,\E ||\tilde{G} (s+k\omega)-\tilde{G} (s)||^{2}ds\\
    & = 2   M^{2} \int_{0}^{\omega}  \sum_{k=0}^{m-1}e^{-2a(t-(k+1)\omega )}\,\E ||G(s+(n+k)\omega)-\tilde{G} (s)||^{2}ds\\
\end{align*}
because $$  \E ||\tilde{G} (s+k\omega)-\tilde{G} (s)||^{2} = 0 \quad   \forall k \geq 0.$$

Since $$\lim_{n\to +\infty} \E ||G(s+(n+k)\omega)-\tilde{G} (s)||^{2} = 0, $$ for $s \in [0,\omega]$, $k\geq 0$ with 
$ \E ||G(s+(n+k)\omega)-\tilde{G} (s)||^{2} \leq 2 K $  \\  and using the fact that 
$$\sum_{k=0}^{m-1}e^{-2a(t-(k+1)\omega )} \leq \frac{1}{1-e^{-2a\omega}} \; \forall t \geq 0,$$ 
it follows by Lebesgue's  dominated convergence theorem that  : 

$$ \lim_{n\to +\infty} \int_{0}^{\omega}  \sum_{k=0}^{m-1}e^{-2a(t-(k+1)\omega )}\,\E ||G(s+(n+k)\omega)-\tilde{G} (s)||^{2}ds  = 0 $$
uniformly on  $t \geq 0$.\\

On the other hand,  

$$ \int_{m\omega}^{t}  M^{2}e^{-2a(t-s)}\, \E ||G(s+n\omega)-\tilde{G} (s)||^{2}ds $$

\begin{align*}
   & \leq   \int_{m\omega}^{(m+1)\omega}  M^{2}e^{-2a(t-s)}\, \E ||G(s+n\omega)-\tilde{G} (s)||^{2}ds\\
   & = \int_{0}^{\omega}  M^{2}e^{-2a(t-s-m\omega)}\, \E ||G(s+(n+m)\omega)-\tilde{G} (s+m\omega)||^{2}ds\\
   & \leq M^{2} \int_{0}^{\omega} \E ||G(s+(n+m)\omega)-\tilde{G} (s+m\omega)||^{2}ds\\
   & \leq 2M^{2} \int_{0}^{\omega} \E ||G(s+(n+m)\omega)-\tilde{G} (s)||^{2}ds\\ 
   & +  2M^{2} \int_{0}^{\omega} \E ||\tilde{G} (s+m\omega)-\tilde{G} (s)||^{2}ds\\
   & = 2M^{2} \int_{0}^{\omega} \E ||G(s+(n+m)\omega)-\tilde{G} (s)||^{2}ds\\ 
\end{align*}
because $$ \E ||\tilde{G} (s+m\omega)-\tilde{G} (s)||^{2} = 0 \quad   \forall m \geq 0.$$

Since $$\lim_{n\to +\infty} \E ||G(s+(n+m)\omega)-\tilde{G} (s)||^{2} = 0, $$ for $s \in [0,\omega]$, $m\geq 0$ and 
$$  \E ||G(s+(n+m)\omega)-\tilde{G} (s)||^{2} \leq 2 K,$$ 
again by the Lebesgue dominated convergence theorem, we have : 

$$      \lim_{n \to +\infty}  2M^{2} \int_{0}^{\omega} \E ||G(s+(n+m)\omega)-\tilde{G} (s)||^{2}ds = 0 $$
so  that 
$$ \int_{m\omega}^{t}  M^{2}e^{-2a(t-s)}\, \E ||G(s+n\omega)-\tilde{G} (s)||^{2}ds.$$

In view of the above, it follows that 
$$  \lim_{n\to +\infty} \E \Big|\Big| J(t,n)-\int_{0}^{t} T(t-s)\tilde{G} (s)dB(s)  \Big|\Big|^{2} = 0 $$

uniformly on $t \geq 0$. \\
Therefore $$ \lim_{n \to +\infty} \E || H(t+n\omega) - H^{*}(t) ||^{2}  = 0$$ 
 uniformly in $t \geq 0$ for the stochastic process $H^{*}(t) : \R_{+} \rightarrow  \LL^{2}(\PP, \mat{H})$\\
 defined as above.
 \end{proof}
 
 Now we can establish the main result of this section.
 \begin{theorem}
 \label{main}
 Let $f ,g : \mathbb{R}_{+} \times \LL^{2}(\PP, \mat{H}) \rightarrow \LL^{2}(\PP, \mat{H})$ be  square mean $\omega$ periodic limit processes 
 in $t \geq 0$ uniformly in $X$ for bounded subsets of  $\LL^{2}(\PP, \mat{H})$. Assume that $f,g$ satisfies a  Lipschitz condition, uniformly in $t \geq 0$ :   that is, there exist  constants  $L_f > 0$ and $L_g >0$   such that
  $$   \E ||f(t,X)-f(t,Y)||^{2} \leq L_f\, \E||X-Y||^{2} \quad \forall t \geq 0, \, \forall X,Y \in \LL^{2}(\PP, \mat{H})$$
    $$   \E ||g(t,X)-g(t,Y)||^{2} \leq L_g\, \E||X-Y||^{2} \quad \forall t \geq 0, \, \forall X,Y \in \LL^{2}(\PP, \mat{H}). $$

 If $$  2M^{2} \big(L_{f} \frac{1}{a^{2}} + L_{g} \frac{1}{a}\big)   < 1 $$ then, there is  a unique square mean asymptotically 
 $\omega$-periodic mild solution \\ of problem (\ref{eq: base}). 
 
 \end{theorem}
 
 \begin{proof}
 
 We define the continuous operator $\Gamma$ on the Banach space \\ $ AP_{\omega}(\mathbb{R}_{+}, \LL^{2}(\PP, \mat{H}))$ by 
 $$  (\Gamma X )(t) =  T(t) c_{0} + \int_{0}^{t} T(t-s)f(s,X(s))ds +  \int_{0}^{t} T(t-s)g(s,X(s))dB(s).   $$
  Note that $ T(t) c_{0} $  is in $ C_{0}(\mathbb{R}_{+}, \LL^{2}(\PP, \mat{H})) \subseteq AP_{\omega}(\mathbb{R}_{+}, \LL^{2}(\PP, \mat{H})) $ \\
Now we denote $ F(s) = f(s,X(s))$,  $ G(s) = g(s,X(s))$ \\
 In view of Theorem \ref{composition} if $X \in AP_{\omega}(\mathbb{R}_{+}, \LL^{2}(\PP, \mat{H}))$  then \\ 
 $F, G  \in P_{\omega}L(\mathbb{R}_{+}, \LL^{2}(\PP, \mat{H}))$.\\
  Applying Lemma \ref{lemme2}, Lemma \ref{lemme4} and  Theorem \ref{criteria}, it follows  that 
 $$ \int_{0}^{t} T(t-s)f(s,X(s))ds  \in AP_{\omega}(\mathbb{R}_{+}, \LL^{2}(\PP, \mat{H}))$$  and 
 $$ \int_{0}^{t} T(t-s)g(s,X(s))dB(s) \in  AP_{\omega}(\mathbb{R}_{+}, \LL^{2}(\PP, \mat{H})).$$
 Hence the operator  $\Gamma$ maps the space 
 $ AP_{\omega}(\mathbb{R}_{+}, \LL^{2}(\PP, \mat{H}))$  into itself. \\
 Finally for any $X, Y \in  AP_{\omega}(\mathbb{R}_{+}, \LL^{2}(\PP, \mat{H}))$
 we have  
 
\begin{align*}
& \E || \Gamma X(t)-\Gamma Y(t)||^{2} \\
 & \leq  2M^{2} \big( \int_{0}^{t} e^{-a(t-s)} ds \big)\, \E \int_{0}^{t}
e^{-a(t-s)} ||f(s, X(s))-f(s,Y(s))||^{2}ds \\
 &  + 2M^{2} \E \int_{0}^{t} e^{-2a(t-s)}\, ||g(s,X(s))-g(s,Y(s))||^{2}ds \\
 &\leq  2M^{2}L_{f}\sup_{s\geq 0} \E || X(s)-Y(s)||^{2}  \big( \int_{0}^{t} e^{-a(t-s)} ds \big)^{2}\\
 & + 2M^{2} L_{g}\sup_{s\geq 0} \E || X(s)- Y(s)||^{2}  \int_{0}^{t} e^{-2a(t-s)}ds\\
 & \leq  2M^{2} \big(L_{f} \frac{1}{a^{2}} + L_{g} \frac{1}{a}\big) \sup_{s\geq 0} \E || X(s)-Y(s)||^{2}.\\
\end{align*}
	
This implies that 
\[ || \Gamma X - \Gamma Y||_{\infty}^{2} \leq 2M^{2} \big(L_{f} \frac{1}{a^{2}} + L_{g} \frac{1}{a}\big) 
|| X- Y ||_{\infty}^{2}.  \]

Consequently, if $ 2M^{2} \big(L_{f} \frac{1}{a^{2}} + L_{g} \frac{1}{a}\big)  < 1$  then $\Gamma$ is a contraction mapping.\\
The proof is completed by using the well-known  Banach fixed-point theorem. \\
 \end{proof}
 
\begin{center}
	\section{An illustrative example}		
\end{center} 
In order to illustrate usefulness of the theoretical results established in the preceding section, we consider the following one-dimensional stochastic heat equation with Dirichlet boundary conditions :
\begin{equation}
\label{eq: example}
\left\{
\begin{array}{l}
d u(t,x) =   \frac{\partial^2 u(t,x)}{\partial x^2} dt  +  f(t,u(t,x)) dt   + g(t,u(t,x)) dB(t) \\
u(t,0)=u(t,1)=0, t\in \mathbb{R}^+,\\
u(0,x)= h(x), \; x \in [0,1].
\end{array}
\right.
\end{equation}
where $B(t)$ is a two-sided  standard one-dimensional Brownian motion  defined on the filtered probability space $( \Omega, \FF, \FF_{t}, \PP) , $   $h \in L^2[0, 1]$ and the functions $f$ and $g$  are defined as  

$$   f(t,u(t,x)) = u(t,x) \psi(t) \quad \textrm{and} \quad  g(t,u(t,x)) = u(t,x) \phi(t),      $$
where $\psi$  and $\phi$ are $\omega$-periodic limit (function) deterministic processes.
Clearly both the  $f$ and $g$ satisfy the Lipschitz conditions with  $L_f = ||\psi||_{\infty} $ and $ L_g =  ||\phi||_{\infty}$.
For instance,  $\psi(t)$ and $\phi(t) $ can be chosen  to be equal to the following  $2$-periodic limit (function) deterministic process (see \cite{xie1})  given by 
\begin{equation}
a_{\{k_{n}\}}(t) =  \left\{
\begin{array}{ll}
1, & t =  2n-1, n \in \mathbb{N} \\
0, & t \in \{ 0, 2\} \cup \{ 2n+1 - k_{n}\}\cup \{ 2n+1 + k_{n}\}\\
\textrm{linear},  &  \; \textrm{in} \; \textrm{between}. 
\end{array}
\right.
\end{equation}
where  $ \{ k_{n} \} \subset ]0,1[$ such that $k_{n} > k_{n+1}$, $k_{n} \to 0$ as $n \to +\infty$.\\
Note that if we define $b : \R_{+} \to  \R $ by 
\begin{equation}
b(t) =  \left\{
\begin{array}{ll}
1, & t =  2n-1, n \in \mathbb{N} \\
0, &  \textrm{otherwise}. 
\end{array}
\right.
\end{equation}
then we  have $b(t) = \lim_{m \to +\infty} a_{\{k_{n}\}}(t+2m)$ so $a_{\{k_{n}\}}$ is a $2$-periodic limit (function) deterministic process.

 Define
$$\mathcal{D}(A)=\{v \;  \textrm{continuous} / v'(r)\; \textrm{absolutely continuous on } [0,1],  \;  v''(r) \in L^{2}[0, 1]$$
$$\, and\; v(0)= v(1)=0\}$$
$$A v =v''\,for\,all\, v \in \mathcal{D}(A).$$
Let $\phi_n(t)=\sqrt{2}\sin(n\pi t)$ for all $n \in \mathbb{N}$. $\phi_n$ are eigenfunctions of the operator $(A, \mathcal{D}(A))$ with eigenvalues $\lambda_n=-n^2$. 
Then, $A$ generates a $C_0$ semigroup  $(T(t))$ of the form
$$T(t)\phi = \sum_{n=1}^{\infty} e^{-n^2  \pi^{2} t}\langle \phi, \phi_n \rangle \phi_n, \;\forall \phi \in L^2[0, 1]$$
and 
$$\vert\vert T(t) \vert\vert \le e^{-\pi^{2} t}, \; for\,  all \,t\ge0 $$ 
Thus $ M=1$  and $a = \pi^{2}$. \\

The equation (\ref{eq: base}) is of the form
\begin{equation*}
	\left\{
	\begin{array}{l}
		dy(t) =  A y(t)dt + f(t,y(t))dt  + g(t,y(t))dB(t) , \\
		y(0)= c_{0}.
	\end{array}
	\right.
\end{equation*}
By using Theorem \ref{main}, we claim that
\begin{theorem}
	If $ ||\psi||_{\infty} + ||\phi||_{\infty}\pi^{2} <  \pi^{4}  / 2$ then the equation	(\ref{eq: example}) admits a unique square mean asymptotically 
	$\omega$-periodic mild solution.
\end{theorem}	

\subsection*{Conflict of Interests}
The authors declare that there is no conflict of interest regarding the publication of this paper.
\subsection*{Acknowledgements}
We thank the anonymous reviewers for their careful reading and their many insightful comments and suggestions.

 \end{document}